\pdfoutput=1
\documentclass[a4paper]{article}
\usepackage[english]{babel}
\usepackage{amsmath,amssymb,amsthm}
\usepackage{enumerate}
\usepackage{ifpdf}
\usepackage{cite}

\ifpdf
\usepackage[bookmarks=false,
pdfstartview=FitH,linkbordercolor={0.5 1 1},
citebordercolor={0.5 1 0.5},unicode,
hyperindex]{hyperref}
\else
\fi

\newtheorem{theorem}{Theorem}[section]
\newtheorem{lemma}[theorem]{Lemma}

\newtheoremstyle{myremark}
{6pt}
{6pt}
{\rmfamily}
{}
{\bfseries}
{.}
{.5em}
{}

\newcommand{\rd}{\mathrm{d}}            

\DeclareMathOperator{\Real}{Re}  
\title{Heavy-Tailed Branching Random Walks on Multidimensional Lattices. A Moment Approach\thanks{Lomonosov Moscow State University, Moscow, Russia,
e-mails: rytova.anastasiya@gmail.com, yarovaya@mech.math.msu.su}}

\author{Anastasiya Rytova, Elena Yarovaya}

\date{}
\begin{document}
\maketitle

\begin{abstract}
We study a continuous-time branching random walk on the lattice $\mathbb{Z}^{d}$, $d\in \mathbb{N}$, with a single source of branching, that is the lattice point where the birth and death of particles can occur. The random walk is assumed to be homogeneous, symmetric and irreducible but, in contrast to previous investigations, the random walk transition intensities $a(x,y)$ decrease as $|y-x|^{-(d+\alpha)}$ for $|y-x|\to \infty$, where $\alpha\in(0,2)$,  that leads to an infinite variance of the random walk jumps. The~mechanism of the birth and death of particles at the source is governed by a continuous-time Bienaym\'e-Galton-Watson
branching process.
The source  intensity is characterized by a certain parameter $\beta$. We calculate the
long-time asymptotic behaviour for all integer moments for the number of particles
at each lattice point and for the total population size.
With respect to the  parameter $\beta$ a non-trivial critical point $\beta_c>0$ is
found for every $d\geq 1$. In particular, if $\beta>\beta_{c}$ the evolutionary operator generated a behaviour of the first moment for the number of particles   has  a positive eigenvalue.
The existence of a positive eigenvalue  yields an exponential growth in $t$ of the particle numbers in the  case $\beta>\beta_c$ called \emph{supercritical}.
Classification of the branching random walk treated as \emph{subcritical} ($\beta<\beta_c$) or \emph{critical} ($\beta=\beta_c$) for the heavy-tailed random walk jumps is more complicated than for a random walk  with a finite variance of jumps.
We study the asymptotic behaviour of all integer moments  of a number of particles  at any point $y\in\mathbb{Z}^d$ and of the
particle population on $\mathbb{Z}^d$ according to
the ratio $d/\alpha$.
\end{abstract}

\section{Introduction}
\label{intro}
We study a continuous-time branching random walk (BRW) on the lattice
${\mathbb Z}^d$ with a single \emph{source} of branching, that is the lattice point where the birth and death of particles can occur. Informally, the process can be described as follows.
Suppose that initially there is a single particle at a point $x\in{\mathbb Z}^d$,
which then performs a random walk on ${\mathbb Z}^d$ until reaching the
source $x_0=0$. At the origin, the particle waits an exponentially distributed time and then  jumps to another lattice point distinct from the origin or dies or gives a random number of offspring.

In case no
branching has occurred, the particle continues its walk until the next return
to the source, and so on. Each of the newborn particles evolves according
to the same rule, independently of the others.
The simplest model with one source of branching, first was introduced in \cite{Y91:e} and
after that reconsidered in a more general setting in a series of publications, see, e.g. \cite{ABY98-1:e, ABY00:e}. Similarly, in \cite{Y2013, Y18-TPA:e, KY} there were
considered BRWs with finitely many sources $x_1,\dots,x_N$.
BRW models is of principal interest, especially in cases when  the branching environment is spatially
inhomogeneous and  the phase space ${\mathbb Z}^d$ is
unbounded, see, e.g. \cite{Y91:e}.
It is well-known in the theory of random media that
inhomogeneity of the environment plays a key role in the formation of
anomalous properties of transport processes which is conventionally expressed in terms of the so-called intermittency concept. For the detailed discussion of the concept of
intermittency of evolutionary random fields, and in particular the concept
of ``strong centres'' of the field generation, see, e.g., \cite{ABMY00:e} and the bibliography therein.
In this context, the model of BRW in the presence of a single source of branching
(or, more generally, with finitely many sources) may be viewed as taking
into account principal ``perturbation'' terms.
On the other hand, the non-compactness of the space destroys
the pure-point spectrum of the operator associated with the process~\cite{Y91:e,ABY00:e,YarBRW:e},
 so that the conventional spectral
methods are no more
readily applicable~\cite{YarBRW:e}.

In the paper, it is assumed that the branching mechanism at the source is governed by a continuous-time Bienaym\'e-Galton-Watson
branching process. It is also assumed that the underlying random walk is homogeneous, symmetric and irreducible but, in contrast to previous investigations, the random walk transition intensities $a(x,y)$ decrease as $|y-x|^{-(d+\alpha)}$ for $|y-x|\to \infty$, where $\alpha\in(0,2)$,  that leads to an infinite variance of the random walk jumps. For such a BRW (called a heavy-tailed BRW) the precise asymptotic behaviour of  all integer moments of the number of particles at each lattice point  and on the entire lattice is remained unexplored in the case when
the positive pure-point spectrum of the operator is empty. Different models of  heavy-tailed random walks without branching of particles were investigated by many authors, see, e.g.,~\cite{bor} and the bibliography therein. The random walk transition probabilities $p(t,x,y)$ for symmetric heavily tailed random walks were investigated in~\cite{MR3349977} under an appropriate regularity condition for $|y-x|+t \to \infty$. The results for the random walk transition probabilities  $p(t,x,y)$ without any regularity conditions, but for fixed space coordinates were obtained in~\cite{RY16:MN}.
So the main goal of the present work is to find the long-time asymptotic behaviour of
all integer moments for the number of particles at each lattice point and
for the total population size, for the case of a heavy-tailed BRW with one source of branching. A BRW with the sources of branching at every lattice point is considered, e.g.,
in~\cite{MR3685637}. Some applications of the theory of heavy-tailed processes can be found, e.g., in~\cite{ZKD15}. Below, we achieve the goal by studying the chain
of the backward differential-difference equations for the moments. The difference operator involved in each moment equation is of the form
$\mathcal{H}_{\beta}:=\mathcal{A}+\beta\delta_0(\cdot)$, where $\mathcal{A}$ is the random walk operator, and the delta-term reflects the presence of the branching
source. The spectral properties of the operator $\mathcal{H}_{\beta}$ play an essential role
in the problem. With respect to a certain parameter $\beta$
characterizing intensity of the source, a non-trivial critical point $\beta_c>0$ is
found for every $d\geq 1$ (unlike a BRW with a finite variance of jumps for which $\beta_c>0$ only for $d\geq 3$) see, e.g.~\cite{Y13-CommStat:e}, which is related to the existence of a positive eigenvalue $\lambda_0$ of the
operator $\mathcal{H}_{\beta}$. The probabilistic meaning of the parameter
$\lambda_0$ is that it determines the rate of the process exponential
growth~\cite{KY}, and hence the case
$\beta>\beta_{c}$ corresponds to the supercritical regime.

Classification of the branching random walk as subcritical or critical for the case of heavy-tailed random walk jumps is more complex than for a random walk  with a finite variance of jumps~\cite{RytYar19}.
We study how the asymptotic behaviour, as $t\to\infty$, of the moments of the particle population and of the number of particles  at the point $y\in\mathbb{Z}^d$ depends on the parameters of walking, branching, and the ratio $d/\alpha$.
The various combinations of which allow classifying the BRW as subcritical, critical or supercritical.
Here, we  give also a complete proof  of the results for subcritical and  critical cases
announced in~\cite{RytYar19}.

The structure of the work is as follows. In \S 2, a formal description of BRW with one branching source and main previous results are reminded. In \S 3, we investigate the structure of the non-negative discrete spectrum of a BRW, lemma~\ref{lem316}, and find the limits of the mean of the local number of particles, theorem~\ref{th331}. In \S 4, theorem~\ref{th523} about asymptotics of all moments of particle number for critical BRW is obtained. At last, in \S 5, we study the asymptotic behaviour of all moments of particle number for subcritical BRW, theorem~\ref{eqlemma532}.

\section{Model and Previous Results}\label{sect:Model}

We consider a stochastic process for a system of particles when every particle can walk on the lattice points and give offsprings or die at the origin (called the \emph{source} of branching).

The underlying random walk is determined by a matrix of transition intensities $A = (a(x, y))_{x,y\in\mathbb{Z}^d}$, where $a(x, y)\geq0$ for $y\not=x$, $-\infty<a(x, x)<0$ and $\sum_{y\in\mathbb{Z}^d}a(x,y)=0$. By the axiomatics from~\cite{GS}, the corresponding transition probability $p(h, x, y)$ of a particle movement from point $x$ to point $y$ over a short period of time $h$, satisfies the equations
\begin{align*}
p(h, x, y) &= a(x, y)h + o(h)\qquad\ \text{for}\ y\not=x,\\
p(h, x, x) &= 1 + a(x, x)h + o(h).
\end{align*}
Such relations imply that the Kolmogorov's backward equations hold
\begin{equation}
\label{backward}
\frac{\rd p(t, x, y)}{\rd t} = \sum_{x'\in\mathbb{Z}^d}a(x, x')p(t, x', y), \quad p(0, x, y) = \delta_y(x),
\end{equation}
where $\delta_y(\cdot)$ is the discrete Kronecker delta-function on $\mathbb{Z}^d$.

We consider a symmetric random walk with $a(x, y) = a(y, x)$. Then, due to the sufficient condition of the boundedness of linear operators given in~\cite[lemma~1]{Y2013},
the expression
\begin{equation}\label{E:defA}
(\mathcal{A}u)(x):=\sum_{x'\in\mathbb{Z}^d}a(x, x')u(x'), \quad u\in\ell^p(\mathbb{Z}^d),
\end{equation}
defines the bounded linear operator $\mathcal{A}:\ell^p(\mathbb{Z}^d)\rightarrow\ell^p(\mathbb{Z}^d)$ for any $p\in[1, \infty]$. Note that the operator $\mathcal{A}$ is self-adjoint in $\ell^2(\mathbb{Z}^d)$. Now we can rewrite equations~\eqref{backward} as differential equations in Banach spaces
\begin{equation}\label{backward2}
\frac{\rd p(t, x, y)}{\rd t} = (\mathcal{A}p(t, \cdot, y))(x), \quad p(0, x, y) = \delta_y(x).
\end{equation}

We assume that $a(x, y) = a(0, y-x)$ which can be treated as a spatial homogeneity of the random walk. Then the values of $a(x, y)$ can be expressed by a function of one argument as follows: $a(x, y) = a(y-x)$, where $a(z):=a(0, z)$, $z\in\mathbb{Z}^d$.
We assume also that the random walk is irreducible which means that for every $z\in\mathbb{Z}^d$ there exists a set of vectors $z_1,\dots,z_k\in\mathbb{Z}^d$ such that $z=\sum_{i=1}^kz_i$ and $a(z_i)\not=0$ for $1\leq i\leq k$.

At last, for the random walk, we assume that for all $z\in\mathbb{Z}^{d}$ with sufficiently large norm the following asymptotic relation holds
\begin{equation}\label{E:infindisp}
a(z)\sim\frac{H\left(\frac{z}{|z|}\right)}{|z|^{d+\alpha}},\quad\alpha\in(0,2),
\end{equation}
where $H(\cdot)$ is a continuous positive symmetric function on the
unit sphere $\mathbb{S}^{d-1}$. This implies that the variance of jumps of the~random walk
\begin{equation*}
\sigma^2 = \sum_{z\not=0}|z|^2\frac{a(z)}{-a(0)}
\end{equation*}
becomes infinite (see~\cite{Y13-CommStat:e})  which means that the random walk under consideration has heavy tails.

The asymptotics of the solution of the Cauchy problem~\eqref{backward2} was found in~\cite[theorem~2.1.1]{YarBRW:e} for the case of a finite variance of jumps
\[
p(t, x, y) \sim \gamma_d  t^{-d/2},\qquad t\to\infty,
\]
with $\gamma_d>0$.
In~\cite[theorem~2]{RY16:MN}, it was proved that under the assumption~\eqref{E:infindisp} the probability $p(t, x, y)$ has the following asymptotics:
\begin{equation}
p(t, x, y) \sim h_{\alpha, d}  t^{-d/\alpha},\qquad t\to\infty, \label{eq2113}
\end{equation}
where $h_{\alpha, d}>0$. Moreover,
\begin{equation}
\label{p_p}
p(t,0,0) - p(t,x,0) \sim \frac{\tilde\gamma_{d,\alpha}(x)}{t^{\frac{d+2}{\alpha}}},\qquad t\to\infty,
\end{equation}
where $\tilde\gamma_{d,\alpha}(x)>0$, see~\cite[theorem 1]{RY19:CommStat:e}.

Now, let us turn to assumptions concerning the branching component of the process.
We assume that branching process at the origin is covered by the continuous-time Bienaym\'e-Galton-Watson branching process having the following infinitesimal generation function
\[
f(u) = \sum_{n=0}^{\infty}b_nu^n, \quad 0\geq u\geq1,
\]
where $b_n\geq0$ for $n\not=1$, $b_1<0$ and $\sum_{n=0}^{\infty}b_n=0$. Assume also $\beta^{(r)}:=f^{(r)}(1)<\infty$ for every $r\in\mathbb{N}$.
Let us denote  \emph{intensity of the source} by
\[
\beta = f'(1) = \sum_{n=1}^{\infty}nb_n = (-b_1)\left(\sum_{n\not=1}n\frac{b_n}{(-b_1)}-1\right),
\]
where the last sum is the mean number of offsprings. According to the imposed assumptions, the probability of producing $n\not=1$ particles from a single one has the following form
\[
p_{*}(h, n) = b_n h + o(h),\qquad h\to0,
\]
where the value $n=0$ means that the particle died.

Now we combine the walking and branching components of the process. At the origin (where the source of branching is located), it is possible that during a short period of time $h$ the particle jumps to a point $y\not=0$ with the transition probability $p(h, 0, y)$, or produces $n\not=1$ particles with the probability $p_{*}(h, n)$, where $n=0$ means the particle death.
Otherwise the particle remains at the origin during the period of time $h$ with the probability
\[
p_{\text{brw}}(h, x, x) = 1 + a(0)h + \delta_0(x)b_1h + o(h),\qquad h\to0.
\]

In what follows, the main  object of interest will be the long-time asymptotic behaviour of the local particle number  at an arbitrary point $\mu_t(y)$, $y\in\mathbb{Z}^d$, and the total population size $\mu_t=\sum_{y\in\mathbb{Z}^d}\mu_t(y)$.  We will investigate the asymptotic behaviour of the moments $m_n(t, x, y)=\mathsf{E}_x \mu_t^{n}(y)$ and $m_n(t, x)=\mathsf{E}_x \mu_t^{n}$, where $\mathsf{E}_x$ is the expectation under condition that there was only one particle in the system at the initial moment of time,  and it located at the point $x\in\mathbb{Z}^d$.

Introduce the linear operator
\begin{equation}\label{E:defHb}
\mathcal{H}_{\beta} = \mathcal{A} + \beta\Delta_0
\end{equation}
on $\ell^p(\mathbb{Z}^d)$, $p\in[1, \infty]$, where the operator $\mathcal{A}$ is defined by~\eqref{E:defA}, and the operator $\Delta_0$ is defined by $(\Delta_0u)(x):=\delta_0(x)u(x)$, where $u\in\ell^p(\mathbb{Z}^d)$.
Then the equations for the first moments are as follows, see, e.g., \cite[theorem 1.3.1]{YarBRW:e},
\begin{alignat}{2}
\label{eq5125}
\frac{\rd m_1(t,x)}{\rd t} &= (\mathcal{H}_{\beta} m_1(t, \cdot))(x),& \quad m_1(0,x) &\equiv 1,\\
\label{eq337}
\frac{\rd m_1(t,x,y)}{\rd t} &= (\mathcal{H}_{\beta} m_1(t, \cdot, y))(x),& \quad m_1(0, x, y) &= \delta_y(x).
\end{alignat}
We would like to point out that equations~\eqref{eq5125} and \eqref{eq337} were obtained in~\cite{YarBRW:e} for the case of a finite variance of jumps, however their derivation  and the final forms remain unchanged for the case of heavy tails as the differential and integral equations below.

Using the notation $\{m_i(t)\}_{i=1}^{n-1}$ for the sets of functions
\[
\{m_i(t, x, y)\}_{i=1}^{n-1}\quad\text{and}\quad\{m_i(t, x)\}_{i=1}^{n-1},
\]
we can write the following differential equations for the higher-order moments, $n\geq2$
\begin{equation*}
\frac{\rd m_n(t)}{\rd t} = (\mathcal{H}_{\beta}m_n(t))(x) + (\Delta_0g_n(m_1(t),\dots,m_{n-1}(t)))(x),
\end{equation*}
where
\begin{equation}
g_n(m_1,m_2,\dots,m_{n-1}) = \sum_{r=2}^n\frac{\beta^{(r)}}{r!} \sum_{\substack{
i_1,\dots,i_r>0\\
i_1+\dots+i_r=n}} \frac{n!}{i_1!\cdots i_r!} m_{i_1}\cdots m_{i_r}. \label{eq135}
\end{equation}
By analogue with the scheme of deriving the integral equations from the differential equations from~\cite{YarBRW:e},
it is possible to derive the integral equations for the first moments
\begin{align}
\label{eq142}
m_1(t, x, y) &= p(t, x, y) + \beta\int_0^tp(t-s, x, 0) m_1(s, 0, y)\rd s,\\
\label{eq143}
m_1(t, x, y) &= p(t, x, y) + \beta\int_0^tm_1(s, x, 0) p(t-s, 0, y)\rd s,\\
\label{eq145}
m_1(t, x) &= 1 + \beta\int_0^tm_1(s, x, 0)\rd s,
\end{align}
and the integral equations for the higher-order moments, $n\geq2$, in forms
\begin{align}\notag
m_n(t, x, y) &= m_1(t, x, y)\\
&\quad + \int_0^tm_1(t-s, x, 0) g_n(m_1(s,0,y),\dots, m_{n-1}(s,0,y))\rd s, \label{eq146}\\
m_n(t, x) &= m_1(t, x) + \int_0^tm_1(t-s, x, 0) g_n(m_1(s,0),\dots, m_{n-1}(s,0))\rd s.\label{eq147}
\end{align}

To investigate the solutions of such equations, the Laplace transform
\begin{equation}
\label{Gf}
G_{\lambda}(x, y) = \int_0^{\infty}e^{-\lambda t}p(t, x, y)\rd t,\qquad \lambda\geq0,
\end{equation}
of transition probability can be used. The function $G_{\lambda}(x, y)$, which is conventionally called the Green function of the transition probabilities, can be expressed, see~\cite{YarBRW:e}, also as
\[
G_{\lambda}(x, y) = \frac1{(2\pi)^d}\int_{[-\pi, \pi]^d}\frac{e^{i\langle \theta, y-x\rangle}}{\lambda-\phi(\theta)}\rd\theta,
\]
where
\[
\phi(\theta) = \sum_{z\in\mathbb{Z}^d}a(z)e^{i\langle \theta, z\rangle}, \qquad \theta\in[-\pi, \pi]^d,
\]
is the Fourier transform of the transition intensity $a(z)$.

\section{Spectral Analysis of the Evolution Operator}
\label{sect:SpAn}

Due to equations~\eqref{eq5125} and \eqref{eq337}, the spectrum of the operator $\mathcal{H}_{\beta}$ defined by~\eqref{E:defHb} determines the asymptotic behaviour of the first moment of particle number.

By~\cite[lemma~3.1.3]{YarBRW:e}, the proof of which does not depend on the conditions for variance of jumps, the number $\lambda$ is an eigenvalue of the operator $\mathcal{H}_{\beta}$ with an eigenvector $f\in\ell^2(\mathbb{Z}^d)$  if and only if  the following conditions are satisfied
\begin{equation}
f(0)\not=0,\quad \beta\not=0, \notag
\end{equation}
\begin{equation}
\int_{[-\pi,\pi]^d}|\lambda-\phi(\theta)|^{-2} \rd\theta < \infty, \label{eq3111}
\end{equation}
\begin{equation}
\beta I_0(\lambda) =1 \label{eq3112},
\end{equation}
where
\begin{equation}
\label{I_0}
I_x(\lambda) = \frac1{(2\pi)^d}\int_{[-\pi, \pi]^d}\frac{e^{-i\langle \theta, x\rangle}}{\lambda - \phi(\theta)}\rd\theta,\qquad x\in\mathbb{Z}^d,
\end{equation}
(note that $I_x(\lambda) \equiv G_{\lambda}(x, 0)$, when $\lambda\ge0$).

Moreover, the eigenvector $f$ corresponding to the eigenvalue $\lambda$ is defined by the equality
\[
f(x) = \beta f(0)I_x(\lambda),\quad x\in\mathbb{Z}^d,
\]
and therefore each eigenvalue of the operator $\mathcal{H}_{\beta}$ is simple.

In the next lemma we answer the question which combinations of the parameters $d$, $\alpha$ and $\beta$ imply that the value $\lambda=0$ is an eigenvalue of the operator $\mathcal{H}_{\beta}$.
\begin{lemma}
\label{lem314}
Let $\lambda=0$. If $d/\alpha\in(1/2, 2]$, $d\in\mathbb{N}$, then the condition~\eqref{eq3111} is not valid,
if $d/\alpha\in(2, \infty)$, $d\in\mathbb{N}$, it is valid.
\end{lemma}
\begin{proof}
By theorem~5 from~\cite{Koz:IJARM16} about the asymptotics as $\theta\to0$ of the function
\begin{equation*}
\sum_{z\in\mathbb{Z}^d\setminus\{0\}}a_z(1-\cos\langle z, \theta\rangle),\quad\theta\in\mathbb{R}^d,
\end{equation*}
where $a_z\|z\|^{d+\alpha}\to1$ for $\|z\|\to\infty$, $\alpha\in(0, 2]$, and
$\langle\cdot, \cdot\rangle$ is the standard inner product and $\|\cdot\|$ is the max-norm on $\mathbb{R}^d$, we have
\[
c|\theta|^{\alpha}\leq |\phi(\theta)| \leq C|\theta|^{\alpha}
\]
in a sufficiently small neighbourhood of zero, where $c, C>0$ are some constants. Then the condition~\eqref{eq3111} for $\lambda=0$ is equivalent to the condition
\[
\int_{[-\pi,\pi]^d}|\phi(\theta)|^{-2}\rd\theta<\infty
\]
which (passing to the generalized polar coordinates) is equivalent to
\[
\int_0^{\rho}r^{d-1-2\alpha}\rd r <\infty
\]
for some small $\rho>0$. Taking into account that the last condition is valid only for $d/\alpha\in(2,\infty)$ and hence is not valid for $d/\alpha\in(1/2,2]$, we finalize the proof of the lemma.
\end{proof}

To detect the values of $\lambda\geq0$ at which equation~\eqref{eq3112} is solvable, we investigate the properties of the function $I_0(\lambda)$.
\begin{lemma}
\label{lem315}
For $\lambda>0$, the function $I_0(\lambda)$ in~\eqref{I_0} is defined, continuous, strictly decreases and positive, wherein
\begin{alignat*}{2}
\lim_{\lambda\to\infty}I_0(\lambda) &=0 \quad & \text{for~} d/\alpha&\in(1/2, \infty), \\
\lim_{\lambda\to0}I_0(\lambda) &= +\infty\quad & \text{for~} d/\alpha&\in(1/2,1],\\
\lim_{\lambda\to0}I_0(\lambda) &= G_0(0, 0) < \infty\quad & \text{for~} d/\alpha&\in(1,\infty),
\end{alignat*}
where $\alpha\in(0,2)$, $d\in\mathbb{N}$.
\end{lemma}

\begin{proof}
The proof is based on the study of convergence domain of the integral
\[
\int_{[-\pi,\pi]^d}|\lambda-\phi(\theta)|^{-1}\rd\theta
\]
according to a scheme similar to that of lemma~\ref{lem314}, and on continuity and monotonicity in $\lambda$ of function $(\lambda - \phi(\theta))^{-1}$.
\end{proof}

As follows from lemma~\ref{lem314}, by virtue of the continuity and strict decreasing of the function $I_0(\lambda)$, the spectrum of the operator $\mathcal{H}_{\beta}$ may contain no more than one positive eigenvalue of unit multiplicity. Furthermore, the set of possible values of the intensity of the source $\beta$ can be divided by a certain threshold value $\beta_c$ in such a way that for $\beta<\beta_c$ there will be no positive eigenvalues, and for $\beta>\beta_c$, there is a unique positive eigenvalue of unit multiplicity. In this sense, the value $\beta_c$ of the intensity of the source can be called the \emph{critical} value.

Let us introduce a classification of BRW with one source of branching which depends on the growth behaviour of the population size. As will be explained in
\S~\ref{sect:SpAn},
for certain dimension $d$ and certain value of the parameter $\alpha$ of the underlying random walk, there exist a critical value $\beta_c$ of the intensity of the source such that the exponential growth of the mean population size is possible only for $\beta>\beta_c$. This is a consequence of the fact that an isolated positive eigenvalue in the spectrum of the operator $\mathcal{H}_{\beta}$ exists only for $\beta>\beta_c$ by definition of $\beta_c$.
In view of this, we will call BRW \emph{supercritical} if $\beta>\beta_c$, \emph{critical} if $\beta=\beta_c$, and \emph{subcritical} if $\beta<\beta_c$.

For case of finite variance of jumps the asymptotics of all the moments $m_n(t, x, y)$, $m_n(t, x)$  was found in the book~\cite{YarBRW:e}. The asymptotics of all the moments $m_n(t, x, y)$, $m_n(t, x)$ with no restrictions on the variance of jumps for supercritical BRW with arbitrary finite number of branching sources was found in~\cite{KY}.

Now we summarize the results about the presence of non-negative eigenvalues in the spectrum of the operator $\mathcal{H}_{\beta}$.
\begin{lemma}
\label{lem316}
For the BRW  on $\mathbb{Z}^d$, $d\in\mathbb{N}$, satisfying~\eqref{E:infindisp}, the following statements about $\beta_c$
and non-negative eigenvalues of $\mathcal{H}_{\beta}$ are valid
\begin{itemize}
\item[\rm(i)] if $d/\alpha\in(1/2, 1]$, then $\beta_c=0$; for $\beta\leq\beta_c$ the operator $\mathcal{H}_{\beta}$ has no non-negative eigenvalues; for $\beta>\beta_c$ the operator $\mathcal{H}_{\beta}$ has an eigenvalue $\lambda>0$ of unit multiplicity that is a solution of equation~\eqref{eq3112};
\item[\rm(ii)] if $d/\alpha\in(1,2]$, then $\beta_c=G_0^{-1}(0,0)>0$; for $\beta\leq\beta_c$ the operator $\mathcal{H}_{\beta}$ has no non-negative eigenvalues; for $\beta>\beta_c$ the operator $\mathcal{H}_{\beta}$ has an eigenvalue $\lambda>0$ of unit multiplicity that is a solution of equation~\eqref{eq3112};
\item[\rm(iii)] if $d/\alpha\in(2,\infty)$, then $\beta_c=G_0^{-1}(0,0)>0$; for $\beta<\beta_c$ the operator $\mathcal{H}_{\beta}$ has no non-negative eigenvalues;
for $\beta\geq\beta_c$ the operator $\mathcal{H}_{\beta}$ has an eigenvalue $\lambda$ of unit multiplicity that is a solution of equation~\eqref{eq3112}, wherein $\lambda=0$ for $\beta=\beta_c$ and $\lambda>0$ for $\beta>\beta_c$.
\end{itemize}
\end{lemma}
\begin{proof}
To determine when the conditions \eqref{eq3111} and \eqref{eq3112} from~\cite[lemma~3.1.3]{YarBRW:e} about eigenvalues of the operator $\mathcal{H}_{\beta}$ are satisfied,
we use lemmas~\ref{lem314}~and~\ref{lem315}. Then the condition~\eqref{eq3111} is valid only for $d/\alpha\in(2,\infty)$ wherein the condition~\eqref{eq3111}  is valid for $\lambda>0$ if $\beta>\beta_c$ and for $\lambda=0$ if $\beta=\beta_c$. The equality $\beta_c=G_0^{-1}(0,0)$, when $G_0^{-1}(0,0)<\infty$, is a consequence of the equation~\eqref{eq3112}.
\end{proof}

Now we can evaluate the limiting behaviour as $t\to\infty$ of the average number of particles $m_1(t, x, y)$ at the every point $y\in\mathbb{Z}^d$ when the process starts from $x$. For every $\lambda\in\mathbb{C}$ and $x, y\in\mathbb{Z}^d$, when $G_{\lambda}(x,0)$, $G_{\lambda}(0,y)$ are defined and finite, denote
\[
c(\lambda,x,y) := \frac{G_{\lambda}(x,0)G_{\lambda}(0,y)}{\|G_{\lambda}(x,0)\|^2_{\ell^2(\mathbb{Z}^d)}}.
\]

\begin{theorem}
\label{th331}
Let $m_1(t, x, y)$ be the solution of the Caushy problem~\eqref{eq337}. Then for the BRW on $\mathbb{Z}^d$, $d\in\mathbb{N}$, satisfying~\eqref{E:infindisp}, and every $y\in\mathbb{Z}^d$, the following statements are valid
\begin{itemize}
\item[\rm(i)]
if $\beta>\beta_c$ and $d/\alpha\in(1/2,\infty)$ then
\begin{equation}\label{y3331c1}
\lim_{t\to\infty}m_1(t,x,y)e^{-\lambda t}c^{-1}(\lambda,x,y) =1,
\end{equation}
where $\lambda$ is the solution of equation~\eqref{eq3112};
\item[\rm(ii)]
if $\beta=\beta_c$ and $d/\alpha\in(2,\infty)$ then
\begin{equation}\label{y3331c2}
\lim_{t\to\infty}m_1(t,x,y)c^{-1}(0,x,y)=1,
\end{equation}
and $m_1(t,x,y)$ is monotonically non-increasing;
\item[\rm(iii)]
if $\beta\leq\beta_c$ and $d/\alpha\in(1/2,2]$ or if $\beta<\beta_c$ and $d/\alpha\in(2,\infty)$ then
\begin{equation}\label{y3331c3}
\lim_{t\to\infty}m_1(t,x,y)=0.
\end{equation}
\end{itemize}

Moreover, the function $m_1(t, x, y)$ converges to zero as $t\to\infty$ uniformly in $x\in~\mathbb{Z}^d$, and it is monotonically non-increasing.
\end{theorem}
\begin{proof}
By~lemma~\ref{lem316} the operator $\mathcal{H}_{\beta}$ has a simple non-negative eigenvalue $\lambda$. Then applying lemmas 3.3.2, 3.3.3 and 3.3.4 from~\cite{YarBRW:e}  (about asymptotics of a solution of the Cauchy problem~\eqref{eq337} with a bounded self-adjoint operator) we readily obtain limits \eqref{y3331c1}, \eqref{y3331c2} and \eqref{y3331c3}.

By analogue with~\cite[lemma~3.3.5]{YarBRW:e},
if the spectrum  $\mathcal{H}_{\beta}$ has no positive eigenvalues, then the function $m_1(t, x, y)$ is monotonically non-increasing in $t$.
\end{proof}

\section{Critical BRW}
\label{sect:CrBRW}
Let us find the asymptotics as $t\to\infty$ of the mean number of particles $m_1(t, x, y)$ at the point $y\in\mathbb{Z}^d$ when the process started from $x\in\mathbb{Z}^d$ for the cases of recurrent random walks.
\begin{theorem}
\label{th522}
Let $\beta=\beta_c$. Then for every $y\in\mathbb{Z}^d$, $d\in\mathbb{N}$, the solution $m_1(t, x, y)$, $x\in\mathbb{Z}^d$, of the~Cauchy problem~\eqref{eq337} under $d/\alpha\in(1/2, 1]$ satisfies the relation
\[
m_1(t,x,y) \sim h_{\alpha, d} t^{-1/\alpha},\qquad t\to\infty,
\]
where $h_{\alpha, d}$ is defined in~\eqref{eq2113}.
\end{theorem}
\begin{proof}
By statement $\rm(i)$ from lemma~\ref{lem316}, we have $\beta_c=0$. Then the Cauchy problem~\eqref{eq337} is equivalent to the Cauchy problem for transition probabilities~\eqref{backward2}.
Therefore, the solution $m_1(t,x,y)$ has the same asymptotics~\eqref{eq2113}.
\end{proof}

Now we find the asymptotics as $t\to\infty$ of the function $m_1(t, x, y)$ for the cases of transient random walks.
\begin{theorem}
\label{th521}
Let $\beta=\beta_c$. Then for every $y\in\mathbb{Z}^d$, $d\in\mathbb{N}$, the solution $m_1(t, x, y)$, $x\in\mathbb{Z}^d$, of the~Cauchy problem~\eqref{eq337} under $d/\alpha\in(1, \infty)$ satisfies the relation
\[
m_1(t,x,y) \sim C_1(x, y)u_1(t),\qquad t\to\infty,
\]
where the functions $ C_1(x,y)$ and $u_1(t)$ are as follows
\begin{alignat*}{3}
C_1(x,y) &= G_0(x,0)G_0(0,y)\gamma^{-1}_{d,\alpha}\Gamma^{-1}(d/\alpha-1),\quad & u_1(t) &= t^{d/\alpha - 2}\quad & \text{for~} d/\alpha&\in(1,2),\\
C_1(x,y) &= G_0(x,0)G_0(0,y)\gamma^{-1}_{d,\alpha},\quad & u_1(t) &= (\ln t)^{-1} & \text{for~} d/\alpha&=2,\\
C_1(x,y) &= G_0(x,0)G_0(0,y)\gamma^{-1}_{d,\alpha},\quad & u_1(t) &= 1 & \text{for~} d/\alpha&\in(2,\infty),
\end{alignat*}
and $\gamma_{d,\alpha}$ are positive constants.
\end{theorem}
\begin{proof}
We find the asymptotics of $m_1(t,x,y)$ as $t\to\infty$ by evaluating the asymptotics of its Laplace transform $\widehat{m}_1(\lambda,x,y)$ as $\lambda\to0$, using the Tauberian theorems (see~\cite[Ch. XIII]{Feller}).
For $\beta\leq\beta_c$ by~\cite[lemma 5.1.3]{YarBRW:e}, the Laplace transform of the solution $m_1(t,x,y)$ of the~Cauchy problem~\eqref{eq337}  for $\Real\lambda>0$ is well defined and can be represented in the following form
\begin{equation*}
\widehat{m}_1(\lambda,x,y) = \frac{\beta G_{\lambda}(0,y) G_{\lambda}(x,0)}{1 - \beta G_{\lambda}(0,0)} +G_{\lambda}(x,y),
\end{equation*}
where $G_{\lambda}(x, y)$ is defined in~\eqref{Gf}.

At first, we consider the case $x=y=0$, then we get
\[
\widehat{m}_1(\lambda,0,0) = \frac{G_{\lambda}(0,0)}{1 - \beta_c G_{\lambda}(0,0)} = \frac{G_0(0,0)G_{\lambda}(0,0)}{G_{\lambda}(0,0) - G_{\lambda}(0,0)}.
\]
Applying the asymptotics of $G_{\lambda}(0,0)$ from~\cite[theorem~1]{Y18-TPA:e} we obtain that $\widehat{m}_1(\lambda,0,0)$ is asymptotically equivalent to $G_0^2(0,0)\gamma^{-1}_{d,\alpha}f(\lambda)$ as $\lambda\to0$, where

\begin{alignat*}{4}
f(\lambda) &= \lambda^{-d/\alpha+1},\quad  &&
\gamma_{d,\alpha} = \Gamma(2-d/\alpha)h_{\alpha,d} \quad &
\text{for~}
& d/\alpha\in(1,2),\\
f(\lambda) &= (\lambda\ln(1/\lambda))^{-1}, \quad  &&
\gamma_{d,\alpha} = -h_{\alpha,d} \quad &
\text{for~}
& d/\alpha=2, \\
f(\lambda) &= \lambda^{-1}, \quad  &&
\gamma_{d,\alpha} = \int_0^{\infty}\left(\int_t^{\infty}p(s,0,0)\rd s\right)\rd t \quad  &
\text{for~}
& d/\alpha\in(2,\infty),
\end{alignat*}
and the constants $h_{\alpha,d}$ are in~\eqref{eq2113}.

Under the condition $\beta=\beta_c$ the function $m_1(t,0,0)$  is monotone by theorem~\ref{th331}. Therefore, we can apply the Tauberian theorems (see~\cite[Ch. XIII]{Feller}) and find that, as $t\to\infty$, the following relations are valid
\begin{alignat*}{2}
m_1(t,0,0) &\sim G_0^2(0,0)\gamma^{-1}_{d,\alpha}(\ln t)^{-1} & \text{for~} &d/\alpha\in(1,2),\\
m_1(t,0,0) &\sim G_0^2(0,0)\gamma^{-1}_{d,\alpha}\Gamma^{-1}(d/\alpha-1)  t^{d/\alpha-2} \quad & \text{for~} & d/\alpha=2,\\
m_1(t,0,0) &\sim G_0^2(0,0)\gamma^{-1}_{d,\alpha} & \text{for~} & d/\alpha\in(2,\infty).
\end{alignat*}

To find the asymptotics of $m_1(t,x,0)$, we express it by the integral equation~\eqref{eq142}, and after that we use statement (e) from~lemma~2 for convolutions in~\cite{Y10-AiT:e}.
Then it remains only to find the asymptotics of $m_1(t,x,y)$, and now we represent it by the integral equation~\eqref{eq143}, apply the same lemma for convolutions and the fact that $\beta=\beta_c=1/G_0(0,0)$ from lemma~\ref{lem316}. Theorem~\ref{th521} is proved.
\end{proof}

Now, we investigate the asymptotic behaviour as $t\to\infty$ of the mean of population size $m_1(t, x)$ when the process started from point $x\in\mathbb{Z}^d$.
\begin{theorem}
\label{th523}
Let $\beta=\beta_c$. Then the solution $m_1(t,x)$, $x\in\mathbb{Z}^d$, $d\in\mathbb{N}$, of the Cauchy problem~\eqref{eq5125} under $d/\alpha\in(1/2, \infty)$ satisfies the relation
\[
m_1(t, x) \sim C_1(x)v_1(t),\qquad t\to\infty,
\]
where the functions $C_1(x)$ and $v_1(t)$ are as follows
\begin{alignat*}{3}
C_1(x) &= 1,\quad & v_1(t) &= 1\quad & \text{for~}&d/\alpha\in(1/2,1],\\
C_1(x) &= G_0(x,0)\gamma_{d,\alpha}^{-1}\Gamma^{-1}(d/\alpha),\quad & v_1(t) &= t^{d/\alpha-1}\quad & \text{for~}&d/\alpha\in(1,2),\\
C_1(x) &= G_0(x,0)\gamma_{d,\alpha}^{-1},\quad & v_1(t) &= t\ln^{-1}t\quad & \text{for~}&d/\alpha=2,\\
C_1(x) &= G_0(x,0)\gamma^{-1}_{d,\alpha},\quad & v_1(t) &= t\quad & \text{for~}&d/\alpha\in(2,\infty),
\end{alignat*}
and $\gamma_{d,\alpha}$ are positive constants.
\end{theorem}
\begin{proof}
For the case $d/\alpha\in(1/2,1]$, we have $\beta=\beta_c=0$ by $\rm(i)$ in lemma~\ref{lem316}. Then equation~\eqref{eq145} implies that $m_1(t,x) \equiv 1$.

For the case $d/\alpha\in(1,\infty)$, we find the asymptotics of $m_1(t, x)$ as $t\to\infty$ by evaluating the asymptotics of its Laplace transform $\widehat{m}_1(\lambda, x)$ as $\lambda\to0$ using the Tauberian theorems (see~\cite[Ch. XIII]{Feller}).
For $\beta\leq\beta_c$ by~\cite[lemma~5.1.4]{YarBRW:e}, the Laplace transform of the solution $m_1(t,x)$ of the Cauchy problem~\eqref{eq5125} for $\Real\lambda>0$ is defined and can be expressed in the form
\begin{equation}
\label{eq5122}
\widehat{m}_1(\lambda,x) = \frac{1 - \beta(G_{\lambda}(0,0) - G_{\lambda}(x,0))}{\lambda(1 - \beta G_{\lambda}(0,0))}.
\end{equation}
Herewith, we have $\beta=\beta_c=G^{-1}_0(0,0)$ by $\rm(ii)$, $\rm(iii)$ in lemma~\ref{lem316}.
Then by asymptotic relations for $G_{\lambda}(0,0)$ from \cite[theorem~1]{Y18-TPA:e}, we get that the asymptotics of $\widehat{m}_1(\lambda,x)$ as $\lambda\to0$ is $G_0(x,0)\gamma_{d,\alpha}^{-1}f(\lambda)$, where
\begin{alignat*}{2}
f(\lambda) &= \lambda^{-d/\alpha}\quad & \text{for~} & d/\alpha\in(1/2,1),\\
f(\lambda) &= \lambda^{-2} \ln^{-1}(1/\lambda) \quad& \text{for~} & d/\alpha=2,\\
f(\lambda) &= \lambda^{-2} \quad & \text{for~} & d/\alpha\in(2,\infty).
\end{alignat*}

As is seen from the integral equation~\eqref{eq145}, the function $m_1(t,x)$ is monotone. Then  applying the Tauberian theorem (see~\cite[Ch. XIII]{Feller}) to $\widehat{m}_1(\lambda,x)$  we find the asymptotics of $m_1(t,x)$ indicated in the statement of the theorem.
\end{proof}

Let us complete the study of the asymptotic behaviour as $t\to\infty$ of the moments of particle number for the critical BRW.

\begin{theorem}
Let $\beta=\beta_c$. Then for the BRW on $\mathbb{Z}^d$, $d\in\mathbb{N}$, under the condition~\eqref{E:infindisp}, and every $x, y\in\mathbb{Z}^d$, $n\in\mathbb{N}$, the following asymptotic relations hold
\begin{equation*}
m_n(t,x,y)\sim C_n(x,y)u_n(t),\quad m_n(t,x)\sim C_n(x)v_n(t),\qquad t\to\infty,
\end{equation*}
where $C_n(x,y)$, $C_n(x)$ are positive constants, and the functions $u_n(t)$, $v_n(t)$ are as follows
\begin{alignat*}{3}
u_n(t) &= t^{-1/\alpha},\quad & v_n(t) &= t^{(1-1/\alpha)(n-1)}\quad &\text{for~}& d/\alpha\in(1/2,1),\\
u_n(t) &= t^{-1},\quad & v_n(t) &= (\ln t)^{n-1}\quad &\text{for~}& d/\alpha=1,\\
u_n(t) &= t^{d/\alpha-2}, & v_n(t) &= t^{(d/\alpha-1)(2n-1)}\quad &\text{for~}& d/\alpha\in(1,3/2),\\
u_n(t) &= t^{-1/2}(\ln t)^{n-1},\quad & v_n(t) &= t^{n-1/2}\quad &\text{for~}& d/\alpha=3/2,\\
u_n(t) &= t^{(d/\alpha-2)(2n-1)+n-1},\quad & v_n(t) &= t^{(d/\alpha-1)(2n-1)}\quad & \text{for~}&d/\alpha\in(3/2,2),\\
u_n(t) &= t^{n-1}(\ln t)^{1-2n},\quad & v_n(t) &= t^{2n-1}(\ln t)^{-2n+1}\quad &\text{for~}& d/\alpha=2,\\
u_n(t) &= t^{n-1},\quad & v_n(t) &= t^{2n-1}\quad &\text{for~}& d/\alpha\in(2,\infty).
\end{alignat*}
\end{theorem}
\begin{proof}
We will find the asymptotics of the moments by induction using the formulas \eqref{eq146}, \eqref{eq147} and \eqref{eq135}.

First we prove the assertions of the theorem for the local moments $m_n(t,x,y)$, $n\geq1$. Denote the convolution
\begin{equation}
W_n(t) := \int_0^tm_1(t-s,x,0)g_n(m_1(s,0,y),\dots,m_{n-1}(s,0,y))\rd s. \notag
\end{equation}
By theorems~\ref{th522} and~\ref{th521}, we have the asymptotics of $m_1(t,x,0)$. Then by the definition~\eqref{eq135}, we find the asymptotics of $g_n(m_1(s,0,y),\dots,m_{n-1}(s,0,y))$, when $n=2$. Now we can find the asymptotics of convolution $W_n(t)$, $n=2$, by~lemma~2 for convolutions from~\cite{Y10-AiT:e}. To find the asymptotics of $m_n(t,x,y)$, $n=2$, it remains only to express one by the formula~\eqref{eq146} and compare the growth rate of $m_1(t,x,y)$ and $W_n(t)$, $n=2$. Following this scheme for $n>2$, we obtain statements of the theorem for $m_n(t,x,y)$, $n>2$.

Using a similar reasoning scheme and applying theorem~\ref{th523}, we find asymptotics of the moments $m_n(t,x)$, $n\geq2$.
\end{proof}

\section{Subcritical BRW}
\label{sect:SbBRW}

Let us find the asymptotic behaviour as $t\to\infty$ of the mean of particle population size $m_1(t, x)$ when the process started from point $x\in\mathbb{Z}^d$.

\begin{theorem}
\label{th531}
Let $\beta<\beta_c$. Then the solution $m_1(t, x)$, $x\in\mathbb{Z}^d$, $d\in\mathbb{N}$, of the Cauchy problem~\eqref{eq5125} under $d/\alpha\in(1/2, \infty)$ satisfies the relation
\[
m_1(t, x) \sim C_1(x)v_1(t),\qquad t\to\infty,
\]
where the functions $C_1(x)$ and $v_1(t)$ are as follows
\begin{alignat*}{3}
C_1(x) &=-(\beta\gamma_{1,\alpha}\Gamma(1/\alpha))^{-1},\quad & v_1(t) &= t^{ 1/\alpha-1}\quad& \text{for~}& d/\alpha\in(1/2,1),\\
C_1(x) &=-(\beta\gamma_{1,1})^{-1},\quad & v_1(t) &= \ln^{-1}t\quad & \text{for~}&d/\alpha=1,\\
C_1(x) &=\frac{1-\beta (G_{0}(0,0)-G_{0}(x,0))}{1-\beta G_{0}(0,0)},\quad &v_1(t) &\equiv1\quad & \text{for~}&d/\alpha\in(1,\infty),
\end{alignat*}
and $\gamma_{1,\alpha}$ are positive constants.
\end{theorem}

\begin{proof}
To find the asymptotics of the function $m_1(t,x)$, we will use the Tauberian theorems (see~\cite[Ch. XIII]{Feller}). By \cite[lemma~5.1.4]{YarBRW:e}, which remains valid in the case of heavy tails~\eqref{E:infindisp}, the Laplace transform of the solution $m_1(t,x)$ of the Cauchy problem~\eqref{eq5125} has the representation by the Green's function \eqref{eq5122}
for $\lambda$ with $\Real\lambda>l$ where $l$ is some non-negative number.

The asymptotics of $G_{\lambda}(0,0)$ as $\lambda\to0$  was found in~\cite[theorem~1]{Y18-TPA:e}. Inspecting its proof, we can observe that the first term of the asymptotics of $G_{\lambda}(x,y)$ with arbitrary $x,y\in\mathbb{Z}^d$ has the same form, because  by~\eqref{eq2113} the asymptotic representation of the transition probability $p(t,x,y)$ is $h_{\alpha,d} /t^{d/\alpha}$, and it does not depend on the lattice points~$x,y$. Then the next relations, as $\lambda\to0$, follow
\begin{alignat*}{2}
\widehat{m}_1(\lambda,x) &\sim -(\beta \gamma_{1,\alpha})^{-1}\lambda^{-1/\alpha}\quad &\text{for~} &d/\alpha\in(1/2,1),\\
\widehat{m}_1(\lambda,x) &\sim (\beta \gamma_{1,1})^{-1} (\lambda\ln\lambda)^{-1}\quad & \text{for~} &d/\alpha=1,
\end{alignat*}
where the constants $\gamma_{d,\alpha}>0$ were defined in~\cite[theorem~1]{Y18-TPA:e}.

Due to the monotonicity of $m_1(t,x)$ by~\eqref{eq145}, we can apply the Tauberian theorem for densities (see~\cite[Ch. XIII, theorem~4]{Feller}) and get the following asymptotic relations as $t\to\infty$
\begin{alignat*}{2}
m_1(t,x) &\sim -(\beta\gamma_{1,\alpha}\Gamma(1/\alpha))^{-1}t^{(1-\alpha)/\alpha}\quad & \text{for~} & d/\alpha\in(1/2,1),\\
m_1(t,x) &\sim -(\beta\gamma_{1,1})^{-1}(\ln t)^{-1}\quad & \text{for~} &d/\alpha=1.
\end{alignat*}

Let us consider the remaining case $d/\alpha\in(1,\infty)$, in which for every $x,y\in\mathbb{Z}^d$ the value of $G_0(x,y)$ is finite. Here by~\eqref{eq5122} we have
\begin{equation*}
\widehat{m}_1(\lambda,x) \sim \frac{1-\beta (G_{0}(0,0)-G_{0}(x,0))}{\lambda (1-\beta G_{0}(0,0))},\qquad \lambda\to0,
\end{equation*}
and therefore, by the Tauberian theorem for densities (see~\cite[Ch. XIII, theorem~4]{Feller}), the asymptotics  of the mean size of particle population is
\begin{equation*}
m_1(t,x) \sim \frac{1-\beta (G_{0}(0,0)-G_{0}(x,0))}{1-\beta G_{0}(0,0)},\qquad t\to\infty.
\end{equation*}
Theorem~\ref{th531} is proved.
\end{proof}

Now we can find the asymptotics as $t\to\infty$ of mean of particle number $m_1(t, x, y)$ at every point $y\in\mathbb{Z}^d$ when process started from $x\in\mathbb{Z}^d$.

\begin{theorem}\label{th533}
Let $\beta<\beta_c$.
Then for the BRW on $\mathbb{Z}^d$, $d\in\mathbb{N}$, satisfying condition~\eqref{E:infindisp}, with representations for the first moments
\begin{equation*}
m_1(t, 0, 0)\sim C_1(0, 0)u_1(t),\quad m_1(t, x) \sim C_1(x)v_1(t),\qquad t\to\infty,
\end{equation*}
for every $x, y\in\mathbb{Z}^d$ the following asymptotic relation hold
\begin{equation}
\label{eqth533}
m_1(t,x,y) \sim C_1(x,y) u_1(t),\qquad t\to\infty,
\end{equation}
where
\begin{alignat*}{3}
C_1(x,y) &= C_1(0,0)g(x)g(y),\quad & u_1(t)&= t^{1/\alpha-2}\quad&\text{for~}& d/\alpha\in(1/2,1),\\
C_1(x,y) &= C_1(0,0)g(x)g(y),\quad & u_1(t)&=t^{-1}\ln^{-2}t\quad&\text{for~}& d/\alpha=1,\\
C_1(x,y) &= (C_1(x) + \beta G_0(y,0)C_1(0)) h_{\alpha,d} \quad  & u_1(t)&=t^{-d/\alpha}\quad&\text{for~}& d/\alpha\in(1,\infty),\\
&\quad+\beta^2 C_1(0,0) G_0(x,0)G_0(y,0),\quad & & & &
\end{alignat*}
and
\begin{equation}
\label{g(x)}
g(x) := 1-\beta\int_0^{\infty}(p(t,0,0)-p(t,x,0))\rd s,
\end{equation}
and $h_{\alpha, d}$ is defined in~\eqref{eq2113}.
\end{theorem}

\begin{proof}
First, we find the asymptotics of $m_1(t, x, 0)$. Substituting $y=0$ into the integral equation~\eqref{eq142}, we obtain
\begin{equation}
\label{m_1(t,x,0)}
m_1(t, x, 0) = p(t, x, 0) + \beta\int_0^tp(t-s, x, 0) m_1(s, 0, 0)\rd s,\\
\end{equation}
where the asymptotics of $p(t, x, 0)$ is known by the relation~\eqref{eq2113}
and the asymptotics of $m_1(t, 0, 0)$ is known by \cite[theorem~3]{R20:FPM}.

Note, that for $d/\alpha\in(1/2, 1]$ due to~lemma~2 for convolutions from~\cite{Y10-AiT:e}, the asymptotics of the integral $\int_0^tp(t-s, x, 0) m_1(s, 0, 0)\rd s$ has the same principal term up to a constant as the asymptotics of $p(t, x, 0)$. Then the sum of the coefficients on the right-hand side of the equation~\eqref{m_1(t,x,0)} satisfies the limit equalities
\begin{equation*}
\lim_{t\to\infty}\left(1+\beta \int_0^{t}m_1(s,0,0)\rd s\right) = \lim_{t\to\infty}m_1(t,0) = 0,
\end{equation*}
due to representation of the $m_1(t, 0)$ by the integral equation~\eqref{eq145} and the limit of $m_1(t, 0)$  by the theorem~\ref{th531}.
But for $d/\alpha\in(1, \infty)$, not all coefficients at principal terms of the asymptotics of the right-hand side~\eqref{m_1(t,x,0)} tend to zero.
For this reason, further we consider the case $d/\alpha\in(1/2, 1]$, when the BRW is recurrent, and the case $d/\alpha\in(1, \infty)$, when the BRW is transient, separately.

For $d/\alpha\in(1/2,1]$, by the integral equation~\eqref{eq142} and symmetry of the random walk, we have
\begin{equation}
\label{eq5320}
m_1(t,x,0) - m_1(t,0,0) = f(t,0,x) + \beta \int_0^{t}f(t-s,0,x) m_1(s,0,0)\rd s,
\end{equation}
where
\begin{equation*}
f(t,z_1,z_2) := p(t,z_1,z_2) - p(t,z_1,0)
\end{equation*}
for any $z_1, z_2\in\mathbb{Z}^d$.
To find the asymptotics of $f(t, 0, x)$, we use the symmetry of the underlying random walk and apply the statement about asymptotics of $p(t,0,0)-p(t,x,0)$ from~\eqref{p_p}.
The asymptotics of $m_1(t,0,0)$ is known by~\cite[theorem~3]{R20:FPM}.  To find the asymptotics of the integral in~\eqref{eq5320}, it remains to use lemma~2 for convolutions from~\cite{Y10-AiT:e}. Then we obtain the equality
\begin{equation*}
\int_0^{t}f(t-s,0,x) m_1(s,0,0)\rd s = (1+\alpha_1(t)) m_1(t,0,0) \int_0^{\infty} f(s,0,x)\rd s,
\end{equation*}
for some $\alpha_1(t)\to0$, $t\to\infty$.
Finally, from equation~\eqref{eq5320}, the asymptotic relation follows
\begin{equation}
\label{eq5321}
m_1(t,x,0) \sim  g(x) m_1(t,0,0),\qquad t\to\infty,
\end{equation}
where $g(x)$ is defined in~\eqref{g(x)}.

By the integral equation~\eqref{eq143} and symmetry of random walk, we can write also
\begin{equation}
\label{eq5321b}
m_1(t,x,y) - m_1(t,x,0) = f(t,x,y) + \beta \int_0^{t}f(t-s,0,y) m_1(s,x,0)\rd s.
\end{equation}
Properties of underlying random walks provide the equality $f(t,x,y) = - (p(t,0,0) - p(t,x-y,0)) + (p(t,0,0) - p(t,x,0))$, and then by the asymptotic relation~\eqref{p_p} we have
\begin{equation}
\label{ftxy}
f(t,x,y) \sim (\widehat{\gamma}_{d,\alpha}(x) - \widehat{\gamma}_{d,\alpha}(x\!-\!y)) t^{-{(d+2)/\alpha}},\qquad t\to\infty.
\end{equation}
We now use the relations~\eqref{eq5321} and \eqref{ftxy} to find the asymptotics of the integral in~\eqref{eq5321b} by~lemma~2 for convolutions from~\cite{Y10-AiT:e}. Then by relation~\eqref{eq5321b}, it follows that
\begin{equation*}
m_1(t,x,y) \sim g(x)g(y) m_1(t,0,0),\qquad t\to\infty,
\end{equation*}
where asymptotics of~$m_1(t,0,0)$ is found in~\cite[theorem~3]{R20:FPM}.

For the case $d/\alpha\in(1, \infty)$, by the integral equation~\eqref{eq142} we express $m_1(t, x, 0)$. We use the asymptotics of $p(t, x, y)$ in \eqref{eq2113} and  the asymptotics of $m_1(t, 0, 0)$ in \cite[theorem~3]{R20:FPM} to find the asymptotics of the integral by~lemma~2 for confolutions from~\cite{Y10-AiT:e}. Then in view of theorem~\ref{th531}, the following equality holds
\begin{align*}
m_1(t,x,0) &= \left(1+(1+a_{d,\alpha}(t,x,0))\beta\int_0^{\infty}m(s,0,0)\rd s\right)p(t,x,0)\\
 &\quad+ \left((1+ b_{d,\alpha}(t,0,0)) \beta\int_0^{\infty}p(s,x,0)\rd s\right)m_1(t,0,0),
\end{align*}
where $a_{d,\alpha}(t,x,0)\to0$, $b_{d,\alpha}(t,0,0)\to0$ as $t\to\infty$ for each $x\in\mathbb{Z}^d$. Then by~\eqref{eq145}
\begin{equation*}
m_1(t,x,0) \sim m_1(t,0)p(t,x,0) + \beta G_0(x,0) m_1(t,0,0),\qquad t\to\infty.
\end{equation*}
In the same way, from the integral equation~\eqref{eq143} we derive
\begin{align*}
m_1(t,x,y) &=\ p(t,x,y) + \left((1+a_{d,\alpha}(t,0,y))\beta \int_0^{\infty}m_1(s,x,0)\rd s\right)p(t,0,y) \\
&\quad+ \left((1+b_{d,\alpha}(t,x,0))\beta\int_0^{\infty}p(s,0,y)\rd s\right)m_1(t,x,0),
\end{align*}
where $a_{d,\alpha}(t,0,y)\to0$, $b_{d,\alpha}(t,x,0)\to0$ as $t\to\infty$ for each $x, y\in\mathbb{Z}^d$. Then, due to the independence of the first term of the asymptotics $p(t,x,y)$  from the coordinates $x,y\in\mathbb{Z}^d$ and by equality~\eqref{eq145}, the asymptotic equality holds
\begin{equation}
m_1(t,x,y) \sim m_1(t, x)h_{\alpha,d}\,t^{-d/\alpha} + \beta G_0(y,0) m_1(t,x,0),\qquad t\to\infty. \notag
\end{equation}
From here  we derive the asymptotic representation \eqref{eqth533} for $m_1(t,x,y)$.
Theorem~\ref{th533} is proved.
\end{proof}

Analysis of the first-order moments $m_1(t, x)$ and $m_1(t, x, y)$ was fulfilled in~theorem~\ref{th531} and theorem~\ref{th533}. In the next theorem we extend the related results,  for the subcritical BRW, to the case of higher-order moments $m_n(t, x)$ and $m_n(t, x, y)$, $n\geq 2$.

\begin{theorem}
\label{eqlemma532}
Let $\beta<\beta_c$.
Then for the BRW on $\mathbb{Z}^d$, $d\in\mathbb{N}$, satisfying condition~\eqref{E:infindisp}
for every $n\geq1$ and every $x, y\in\mathbb{Z}^d$ the asymptotic relations hold
\begin{alignat}{2}
m_n(t, x, y) &\sim C_n(x,y)u_1(t),&\quad t\to\infty,\label{mn(t,x,y)}\\
m_n(t, x) &\sim C_n(x)v_1(t),& \quad t\to\infty. \label{mn(t,x)}
\end{alignat}
Here $C_1(x, y)$, $u_{1}(t)$ and $C_1(x)$, $v_{1}(t)$ are obtained in theorem~\ref{th533}, while for $n\geq2$ we have
\[
C_n(x,y) = C_1(x,y) + C_1(x,0)\int_0^{\infty}g_n(m_1(s,0,y),\dots,m_{n-1}(s,0,y))\rd s,
\]
where for the function $g_n(m_1(t),\dots,m_{n-1}(t))$ defined in~\eqref{eq135}, with the notation $m_i(t)\equiv m_i(t, 0, y)$, $i\in\mathbb{N}$, the following asymptotic relation holds
\begin{multline}
g_n(m_1(t),m_2(t),\dots,m_{n-1}(t))\\ \sim \frac{\beta^{(2)}}{2}\sum_{i=1}^{n-1}\frac{n!}{i! (n-i)!} m_i(t)m_{n-i}(t),\quad t\to\infty.\label{gn(m)}
\end{multline}
Besides, for $n\geq2$ we also have
\begin{alignat*}{2}
C_n(x) &= C_1(x),\qquad & \text{for~} d/\alpha&\in(1/2,1],\\
C_n(x) &= C_1(x) + \chi_n(x)\int_0^{\infty}m_1(s,x,0)\rd s,\qquad & \text{for~} d/\alpha&\in(1,\infty),
\end{alignat*}
where
\begin{equation}
\label{chi0}
\chi_n(x) = \frac{\beta^{(2)}}{2}\sum_{i=1}^{n-1}\frac{n!}{i! (n-i)!}C_i(x)C_{n-i}(x).
\end{equation}
\end{theorem}
\begin{proof}
We use the notation $m_i(t)$, $i\in\mathbb{N}$, to formulate statements that are valid simultaneously for $m_i(t, x, y)$ and  $m_i(t, x)$.
The proof is by induction on $n\geq2$ and consists in
alternate derivation of asymptotics of the functions $g_n(m_1(t), \dots ,m_{n-1}(t))$ and of the moments $m_n(t)$ by the integral equations for moments \eqref{eq146} and \eqref{eq147}.

For $n=2$ by the definition~\eqref{eq135}, we have $g_2(m_1(t)) = \beta^{(2)}m_1^2(t)$.
Then the relation~\eqref{gn(m)} under $n=2$ follows.

This makes it possible to find the asymptotics of $m_2(t, x, y)$ and $m_2(t, x)$ using the integral equations for moments \eqref{eq146} and \eqref{eq147}.
To find asymptotics of convolutions in these equations,
we write down the asymptotics of $u_1(t)$, $u_1^2(t)$ and $v_1^2(t)$ from theorems~\ref{th533} and~\ref{th531}, that is
\begin{alignat*}{4}
u_1(t) &= t^{ 1/\alpha-2},\quad&u_1^2(t) &= t^{ 2/\alpha - 4},\quad&v_1^2(t) &= t^{ 2/\alpha -2}\quad&\text{for~}& d/\alpha\in(1/2,1),\\
u_1(t) &= t^{-1}\ln^{-2}t,\quad&u_1^2(t) &= t^{-2}\ln^{-4}t,\quad&v_1^2(t) &= \ln^{-2}t\quad&\text{for~}& d/\alpha=1,\\
u_1(t) &= t^{-d/\alpha},\quad&u_1^2(t) &= t^{-2d/\alpha},\quad&v_1^2(t) &= 1\quad&\text{for~}& d/\alpha\in(1,\infty),
\end{alignat*}
and apply lemma~2 for convolutions from~\cite{Y10-AiT:e}, namely,
statements (k), (n) to $m_2(t,x,y)$, and (d), (e) to $m_2(t,x)$ for cases $d/\alpha=1$ and $d/\alpha\in(1/2, 1)\cup(1, \infty)$ respectively.
Then the following equalities hold
\begin{align}
\int_0^tm_1(t-s,x,0)&g_2(m_1(s,0,y))\rd s \notag\\
&= (1+a_{d,\alpha}(t))m_1(t,x,0)\int_0^{\infty}g_2(m_1(s,0,y))\rd s, \label{eq1lemma532}\\
\int_0^tm_1(t-s,x,0)&g_2(m_1(s,0)) \rd s \notag\\
&= (1+b_{d,\alpha}(t))g_2(m_1(t,0))\int_0^{\infty}m_1(s,x,0)\rd s,\label{eq2lemma532}
\end{align}
where $a_{d,\alpha}(t)\to0$, $b_{d,\alpha}(t)\to0$, as $t\to0$.
Hence, by the integral equations~\eqref{eq146} and \eqref{eq147}, we obtain the following asymptotic equalities
\begin{alignat*}{2}
m_2(t,x,y) &\sim m_1(t,x,y) + m_1(t,x,0)\int_0^{\infty}g_2(m_1(s,0,y))\rd s,& \quad t\to\infty,\\
m_2(t,x) &\sim m_1(t,x) + g_2(m_1(t,0))\int_0^{\infty}m_1(s,x,0)\rd s,& \quad t\to\infty.
\end{alignat*}
Now we rewrite the asymptotics of the functions above in terms of $C_1(x, y)$, $u_1(t)$, $C_1(x)$, $v_1(t)$ from theorems~\ref{th531} and \ref{th533}, namely
\begin{alignat*}{2}
g_2(m_1(t,x,y)) &\sim \beta^{(2)}C_1^2(x,y)u_1^2(t),& \quad t\to\infty,\\
m_2(t,x,y) &\sim \left(C_1(x,y) + C_1(x,0) \int_0^{\infty}g_2(m_1(s,0,y))\rd s\right)u_1(t),& \quad t\to\infty,\\
g_2(m_1(t,x)) &\sim \beta^{(2)}C_1^2(x)v_1^2(t),& \quad t\to\infty.
\end{alignat*}
Then for function $m_2(t, x)$, we have under $d/\alpha\in(1/2,1]$ the asymptotics
\[
m_2(t,x) \sim C_1(x)v_1(t),\qquad t\to\infty,
\]
and, under $d/\alpha\in(1, \infty)$, the asymptotics
\[
m_2(t,x) \sim \left(C_1(x) + \beta^{(2)}C_1^2(0)\int_0^{\infty}m_1(s,x,0)\rd s\right)v_1(t),\qquad t\to\infty.
\]

Continuing the same reasoning, by alternate applying the formula in definition \eqref{eq135} and the integral equations \eqref{eq146}, \eqref{eq147},  we will obtain the asymptotic representations for $g_n(m_1(t),\dots,m_{n-1}(t))$, $n\geq2$, as polynomials of $u_1(t)$, when $m_n(t)$ is $m_n(t, x, y)$, and as polynomials of $v_1(t)$, when $m_n(t)$ is $m_n(t, x)$. Also we will obtain asymptotic representations of $m_n(t,x,y)$ and  $m_n(t,x)$ as the product of some constant on $u_1(t)$ and $v_1(t)$ respectively.

Assume that for some $n\geq2$ the asymptotic relations \eqref{gn(m)}, \eqref{mn(t,x,y)}, \eqref{mn(t,x)} hold. Then, in particular, we have
\[
m_{n-1}(t,x,y)\sim C_{n-1}(x,y)u_1(t),\quad m_{n-1}(t,x)\sim C_{n-1}(x)v_1(t),\qquad t\to\infty.
\]
Let us prove that these relations are also valid for $n+1$.

Rewrite the formula~\eqref{gn(m)} in the form
\begin{alignat*}{2}
g_n(m_1(t,x,y),\dots,m_{n-1}(t,x,y)) &\sim \frac{\beta^{(2)}}{2}\sum_{i=1}^{n-1}\frac{n!}{i! (n-i)!}C_i(x, y)C_{n-i}(x, y)u_1^2(t),\\
g_n(m_1(t,x),\dots,m_{n-1}(t,x)) &\sim \frac{\beta^{(2)}}{2}\sum_{i=1}^{n-1}\frac{n!}{i! (n-i)!}C_i(x)C_{n-i}(x)v_1^2(t).
\end{alignat*}
Due to the power-logarithmic behaviour of the functions $u_1(t)$ and $v_1(t)$, and by analogue with the scheme of obtaining \eqref{eq1lemma532} and \eqref{eq2lemma532}, we derive by~lemma~2 for convolutions from~\cite{Y10-AiT:e}, as $t\to\infty$, the following asymptotic representations
\begin{align*}
\int_0^tm_1(t-s,&x,0)g_n(m_1(t,0,y),\dots,m_{n-1}(t,0,y))\rd s\\
&= (1+a_{d,\alpha,n}(t))m_1(t,x,0)\int_0^{\infty}g_2(m_1(s,0,y),\dots,m_{n-1}(s,0,y))\rd s,\\
\int_0^{t}m_1(t-s,&x,0)g_n(m_1(s,0),\dots,m_{n-1}(s,0))\rd s\\
&=(1+b_{d,\alpha,n}(t))g_n(m_1(t,0),\dots,m_{n-1}(t,0))\int_0^{\infty}m_1(s,x,0)\rd s,
\end{align*}
where $a_{d,\alpha,n}(t)\to0$, $b_{d,\alpha,n}(t)\to0$ as $t\to\infty$. By analogue with the scheme of obtaining the asymptotic relations for $m_2(t, x, y)$ and $m_2(t, x)$ above and equations~\eqref{eq146} and \eqref{eq147}, as $t\to\infty$, for $d/\alpha\in(1/2, \infty)$, it is possible to obtain the following relations
\[
m_n(t,x,y)\sim \left(C_1(x,y) + C_1(x,0)\int_0^{\infty}g_n(m_1(s,0,y),\dots,m_{n-1}(s,0,y))\rd s\right)u_1(t),
\]
as well as
\begin{alignat*}{3}
m_n(t,x) &\sim C_1(x)v_1(t)
&\qquad \text{for~} d/\alpha\in(1/2,1],\\
m_n(t,x) &\sim \left(C_1(x) + \chi_0(x)\int_0^{\infty}m_1(s,x,0)\rd s\right)v_1(t)
&\quad \text{for~} d/\alpha\in(1, \infty),
\end{alignat*}
where $\chi_0(x)$ is defined in~\eqref{chi0}.

Then the representation
\begin{align*}
g_{n+1}(m_1(t),\dots,m_n(t)) &= \frac{\beta^{(2)}}{2} \sum_{i=1}^{n}\frac{(n+1)!}{i!(n+1-i)!}m_i(t)m_{n+1-i}(t) \\
&\qquad+ \sum_{r=3}^{n+1}\frac{\beta^{(r)}}{r!}\sum_{\substack{
i_1,\dots,i_r>0,\\
i_1+\cdots+i_r=n+1
}} \frac{(n+1)!}{i_1!\cdots i_r!} m_{i_1}(t)\cdots m_{i_r}(t),
\end{align*}
justifies the validity of the relation~\eqref{gn(m)} and simultaneously the validity of the relations~\eqref{mn(t,x,y)} and \eqref{mn(t,x)}, which completes the proof of the theorem.
\end{proof}

\section{Acknowledgement}
The work of the authors is supported by the RFFR (grant 20-01-00487).

\label{lastpage}

\begin{thebibliography}{99}
\small
\setlength{\itemsep}{0pt}

\bibitem{MR3349977}
A.~Agbor, S.~Molchanov and B.~Vainberg.
Global limit theorems on the convergence of multidimensional random walks to stable processes.
\textit{Stoch. Dyn.} \textbf{15} (2015), 1550024, 14.


\bibitem{ABY98-1:e}
    S.~Albeverio, L.~Bogachev and E.~Yarovaya.
Asymptotics of branching symmetric random walk on the lattice with a single source.
\textit{C. R. Acad. Sci. Paris S\'{e}r. I Math.} \textbf{326} (1998), 975--980.

\bibitem{ABY00:e}
   S.~Albeverio, L.~V.~Bogachev and E.~B.~Yarovaya.
    Branching random walk with a single source.
   In \textit{Communications in difference equations ({P}oznan, 1998)},
   9--19 (Amsterdam: Gordon and Breach, 2000).

\bibitem{ABMY00:e}
S.~Albeverio, L.~Bogachev, S.~Molchanov and E.~Yarovaya.
 Annealed moment {L}yapunov exponents for a branching random
                  walk in a homogeneous random branching environment.
\textit{Markov Process. Related Fields.} \textbf{6} (2000), 473--516.

\bibitem{bor}
A.~Borovkov and K.~Borovkov.
\textit{Asymptotic Analysis of Random Walks. Heavy-Tailed Distributions}
(Cambridge University Press, 2008).

\bibitem{Feller}
W.~Feller.
\textit{ An introduction to probability theory and its applications.
  {V}ol. {II}}
(Second edition. John Wiley \& Sons, Inc., New York-London-Sydney, 1971).

\bibitem{MR3685637}
A.~Getan, S.~Molchanov and B.~Vainberg.
Intermittency for branching walks with heavy tails.
\textit{Stoch. Dyn.} \textbf{17} (2017), 1750044, 14.

\bibitem{GS}
I.~I.~Gihman and A.~V.~Skorohod.
\textit{The theory of stochastic processes II}.
(Grundlehren Math. Wiss., 218, Springer-Verlag, New York-Heidelberg, 1975).

\bibitem{Koz:IJARM16}
V.~Kozyakin.
Hardy type asymptotics for cosine series in several variables with
  decreasing power-like coefficients.
\textit{Int. J. Adv. Res. Math.} \textbf{5} (2016), 35--51.


\bibitem{KY}
I.~Khristolyubov and E.~B.~Yarovaya.
A limit theorem for supercritical branching random walks with branching sources of varying intensity.
\textit{Teor. Veroyatnost. i Primenen.} \textbf{64} (2019), 456--480.

\bibitem{R20:FPM}
A.~I.~Rytova.
Harmonic analysis of branching random walks with heavy tails.
\textit{Fundam. Prikl. Mat.} (in print), (2020).


\bibitem{RY16:MN}
A.~I.~Rytova and E.~B.~Yarovaya.
Multidimensional Watson lemma and its applications.
\textit{Math. Notes.} \textbf{99} (2016), 406--412.


\bibitem{RY19:CommStat:e}
A.~Rytova and E.~Yarovaya.
Survival analysis of particle populations in branching random walks.
\textit{Comm. Statist. Simulation Comput.} (published online), (2019), 1--16.

\bibitem{RytYar19}
A.~I.~Rytova and E.~B.~Yarovaya.
Moments of the numbers of particles in a heavy-tailed branching random walk.
\textit{Uspekhi Mat. Nauk.} \textbf{74} (2019), 165--166.


\bibitem{YarBRW:e}
E.~B.~Yarovaya.
\textit{Branching random walks in a heterogeneous environment} (in Russian)
(Moscow: Centre of Applied Investigations of the Faculty of Mechanics and Mathematics of the
  Moscow State University, 2007).

\bibitem{Y10-AiT:e}
E.~B.~Yarovaya.
Models of branching walks and their use in reliability theory.
\textit{Autom. Remote Control.} \textbf{71} (2010), 1308--1324.


\bibitem{Y13-CommStat:e}
E.~Yarovaya.
Branching random walks with heavy tails.
\textit{Commun. Statist. Theory Methods.} (16) 42, (2013), 2301--2310.

\bibitem{Y2013}
E.~B.~Yarovaya.
Branching random walks with several sources.
\textit{Math. Popul. Stud.} \textbf{20} (2013), 14--26.

\bibitem{Y18-TPA:e}
E.~B.~Yarovaya.
Spectral asymptotics of a supercritical branching random walk.
\textit{Theory Probab. Appl.} \textbf{62} (2018), 413--431.

\bibitem{Y91:e}
E.~B.~Yarovaya.
Application of spectral methods in the study of branching
processes with diffusion in a noncompact phase space.
\textit{Theoret. and Math. Phys.} \textbf{88} (1991), 690--694.

\bibitem{ZKD15}
E.~Zhizhina, S.~Komech and X.~Descombes.
Modelling axon growing using CTRW.  (2015),
arXiv preprint arXiv:1512.02603.
\end{thebibliography}
\end{document}